\documentclass[12pt]{amsart}

\usepackage{amsmath,amsthm,amsfonts,amscd,amssymb,euscript}

\swapnumbers

\newtheorem{Lemma}{Lemma}
\newtheorem{Theorem}[Lemma]{Theorem}
\newtheorem{Proposition}[Lemma]{Proposition}

\newtheorem*{Theorem*}{Lemma}
\theoremstyle{remark}
\newtheorem{Remark}[Lemma]{Remark}

\newtheorem*{Remark*}{Remark}
\theoremstyle{definition}
\newtheorem{Definition}[Lemma]{Definition}

\numberwithin{Lemma}{section}

\numberwithin{equation}{section}

\newcommand{\R}{\ensuremath{{\mathbb R}}}
\newcommand{\N}{\ensuremath{{\mathbb N}}}

\newcommand{\E}{\ensuremath{\mathcal E}}
\newcommand{\F}{\ensuremath{\mathcal F}}
\newcommand{\oka}{\ensuremath{\mathcal O}}

\newcommand{\m}{\ensuremath{\mathfrak m}}

\begin{document}

\title[Weierstrass Preparation Theorem]{Failure of the Weierstrass Preparation Theorem in quasi-analytic Denjoy-Carleman rings}

\author[F. Acquistapace, F. Broglia]{Francesca Acquistapace, Fabrizio Broglia}

\address{Dipartimento di Matematica, Universit\`a degli Studi di Pisa, Largo Bruno Pontecorvo, 5, 56127 Pisa, Italy}

\email{acquistf@dm.unipi.it, broglia@dm.unipi.it}

\author[M. Bronshtein]{Michail Bronshtein}

\address{Department of Mathematics, Kazan National Research Technological University, 68 Karl Marx St., Kazan 420015, Republic of Tatarstan, Russian Federation}

\email{bronmich@gmail.com}

\author[A. Nicoara]{Andreea Nicoara}

\address{Department of Mathematics, University of Pennsylvania, 209 South $33^{rd}$ St.,  Philadelphia, PA 19104, United States}

\email{anicoara@math.upenn.edu}

\author[N. Zobin]{Nahum Zobin}

\address{Department of Mathematics, College of William and Mary, P.O. Box 8795, Williamsburg, VA 23187-8795,  United States}

\email{nxzobi@wm.edu}

\thanks{Much of the work presented in this paper was done while the authors were visiting the American Institute of Mathematics, Palo Alto, CA for a SQuaRE workshop. The authors wish to express their gratitude to AIM for the hospitality and financial support.}



\begin{abstract}
It is shown that Denjoy-Carleman quasi-analytic rings of germs of functions in two or more variables fail to satisfy the Weierstrass Preparation Theorem. The result is proven via a non-extension theorem.
\end{abstract}

\maketitle

\tableofcontents
\parindent = 0 pt

\section{Introduction}

Denjoy-Carleman rings of infinitely differentiable functions have been classically studied in Partial Differential Equations, Function Theory and other fields of Analysis. Since the mid 1970's, their algebraic structure, which is necessary for understanding the geometry of zero sets of Denjoy-Carleman functions, has also been attracting quite a lot of attention.
\medskip

\medskip
The technical complication for Denjoy-Carleman quasi-analytic classes of functions is that applying Weierstrass division to a function in a Denjoy-Carleman class might produce a quotient and a remainder outside that class as shown by Childress \cite{childress} in 1976. Since Weierstrass division is the standard method via which the Weierstrass Preparation Theorem is proven, the Weierstrass Preparation Theorem became problematic as well as difficult to establish in the wake of Childress' result. Most experts expected it to fail, but a counterexample turned out to be elusive to construct. Partial results on Weierstrass division in the Denjoy-Carleman classes due to Childress, Chaumat, and Chollet gave a glimpse into what a counterexample to the Weierstrass Preparation Theorem needed to be. A function has the Weierstrass division property if it divides any other function in the ring, and the quotient and remainder of that division are inside the ring. A very elementary argument then shows that every function that has the Weierstrass division property satisfies the Weierstrass Preparation Theorem. By Childress' work in 1976 \cite{childress} (one direction of the equivalence) and Chaumat-Chollet's work in 2004 \cite{chaumatchollet} (the other direction), a normalized Weierstrass polynomial has the Weierstrass division property in a Denjoy-Carleman ring iff it is hyperbolic, i.e. it only has real roots. This indicated that if the Weierstrass Preparation Theorem failed, it did so on functions that did not have the Weierstrass division property, of which there existed some by Childress' work.

\medskip
Further complications, however, arise for two reasons. First of all, failure of Weierstrass division seemingly does not imply the failure of the Weierstrass Preparation Theorem, i.e. these two conditions are not equivalent unlike in the classically known cases of holomorphic and real analytic germs. In fact, we are not aware of any example of a ring of smooth functions for which the Weierstrass Preparation Theorem holds, but the Weierstrass  division fails. With respect to this question of the existence of Weierstrass division, of the Weierstrass Preparation Theorem, and of their relationship for subrings of formal power series, A'Campo gave an interesting treatment in \cite{acampo}. Second of all, standard methods in commutative algebra do not yield any information about the Weierstrass Preparation Theorem.

\medskip
 By contrast, non-quasi-analytic Denjoy-Carleman rings behave much more like rings of smooth functions in the sense that they possess analogs of the Weierstrass (Malgrange) Preparation and Division Theorems as shown in Bronshtein \cite{bronshteinnext}. In establishing such properties of non-quasi-analytic classes, a very important role is played by the fact that functions from these classes can be extended  with a {\em controlled widening} of the class. Various extension results for non-quasi-analytic Denjoy-Carleman classes were proven by Carleson \cite{carleson}, Mityagin \cite{mityagin}, Ehrenpreis \cite{ehrenpreis}, Wahde \cite{wahde}, Zobin \cite{zobin, zobin2} as well as many other authors.

\medskip
In the case of quasi-analytic classes, however, there are important  non-extension results, due to Carleman \cite{carlemanb}, Lyubich-Tkachenko \cite{lyubichtkachenko}, and, more recently, Langenbruch \cite{langenbruch} and Thilliez \cite{thilliez}; see also Nowak \cite{nowakcounter}. 
When working on the problems discussed in this article we discovered that the failure of the Weierstrass Preparation Theorem for a quasi-analytic ring (even if we allow any quasi-analytic widening of this ring) follows (and is actually more or less equivalent) to the failure of the extension property for such classes. Unlike these previous results, we prove a different non-extension theorem by producing a  very  explicit  example of non-extendable function with important additional properties that are potent enough to permit contradicting the Weierstrass Preparation Theorem. Therefore, the main result of this paper is the following:

\medskip
\begin{Theorem}
\label{wpt}
The Weierstrass Preparation Theorem does not hold in any quasi-analytic class (strictly containing the real analytic class)  in dimension $\geq 2.$ Moreover, it does not hold even if we allow the unit and the distinguished polynomial to be in any wider quasi-analytic class.
\end{Theorem}

\medskip
The main theorem is proven via the following non-extension result:

\medskip
\begin{Theorem}\label{misha}
Let  $M=\{m_n\}$ and $\widetilde  M=\{\widetilde m _n\}$ be two sequences.
 Then  there exists  a
  function $f$   with the following properties:
\begin{enumerate}
\item $f\in C^{M}[0,\infty)$.
\item $\forall\, x\in \R,\, \forall\, p\in \N \quad \left| f^{(p)}(x)\right|\leq p!\,{2^{p+1}}\,\widetilde   m_p ,$ so in particular, $f\in C^{\widetilde M}(\R)$.
\item There exists a sequence $ \{x_p\}\subset \R$ with $ x_p <0$ and $\displaystyle\lim_{p\to \infty} x_p=0$  and a sequence $\{n_p\} \subset \N,$ $\displaystyle\lim_{p\to \infty} n_p=\infty$ such that
$$ \forall\, p\quad  \left| f^{(n_p)}(x_p)\right| \geq  n_p!\,\widetilde   m_{n_p}.$$
\item $f$ is analytic outside of any neighborhood of \, $0$.

\end{enumerate}

\end{Theorem}
\medskip

Specifically, this  result shows that if we choose a sequence $K$ such that the class $C^K(\R)$
 is quasi-analytic and $C^K(\R)\subsetneq C^{\widetilde{M}}(\R)$, then there exists a function from the quasi-analytic class $C^M[0,\infty) $ which cannot be extended to a function from $C^K(\R),$ since a  quasi-analytic extension, if it exists, is unique, so it should coincide with this function on the whole axis, and this function does not belong to $C^K(\R).$

\medskip
It should be noted that the function constructed to satisfy the conclusion of this theorem bears a relationship to the very first example of a Denjoy-Carleman function, which is not real analytic, given by  \'{E}mile Borel in  \cite{borel1} and \cite{borel2}; see also  \cite{thilliez}. Furthermore, examining the Taylor expansion at $0$ of $f$ reveals another important fact. $T_0 f \in \F^M,$ so $T_0 f$ is a formal power series satisfying all the derivative bounds given by the sequence $M,$ and yet it corresponds to $f,$ which is not a function in $C^M$ in any neighborhood around the origin as soon as the sequence $\widetilde M$ is suitably chosen; see Theorem \ref{misha2} below. Using a functional theoretic argument that produced a lacunary series, Torsten Carleman showed in \cite{carlemanb} as early as 1926 that there exist non-convergent formal power series whose coefficients satisfy the Denjoy-Carleman bounds but which do not correspond to Denjoy-Carleman functions in the same class. The function $f$ in Theorem~\ref{misha} finally gives an explicit example of a function of this kind.

\medskip
The authors of this work who come from quite different mathematical backgrounds were very fortunate to meet at the American Institute of Mathematics in Palo Alto, CA for a SQuaRE workshop. They wish to express their gratitude to AIM for the hospitality and financial support. They are also indebted to Michel Coste for his help in the proof of Lemma \ref {WEP}, to Edward Bierstone for his suggestions that led to a more elegant proof of the main result, and to the referee for his observations that improved the clarity of this paper.

\medskip The paper  is organized as follows: A  preliminary section is
devoted  to  describing  in  detail  the  classes  of  Denjoy-Carleman
functions   considered as well as the  intricacies  of
Weierstrass division  in the Denjoy-Carleman  settings and its
relation to Weierstrass Preparation. Section 3 deals with extension of
functions in  quasi-analytic classes. In  Section 4, Theorem~\ref{wpt}
is proven assuming  Theorem~\ref{misha}. Finally, Section 5 constructs
the non-extendable function that proves Theorem~\ref{misha}.

\section{Preliminaries}

{\it Denjoy-Carleman classes -- basic definitions and results}. The set-up explained below constitutes the most standard one for Denjoy-Carleman  classes as detailed in \cite{thilliez}.
\medskip

\noindent  For a multi-index $\alpha=(\alpha_1,\ldots, \alpha_n) \in \N^n,$ we will employ the following notation:
$$|\alpha|:=\alpha_1+\cdots+\alpha_n,$$
$$\alpha!:=\alpha_1 !  \cdots \alpha_n !,$$
$$D^\alpha := \frac{\partial^{\, |\alpha|}}{\partial x_1^{\,
\alpha_1} \cdots \,
\partial x_{n}^{\, \alpha_{n}}},$$  $$x^{\alpha}:={x_1}^{\alpha_1} \cdots {x_n}^{\alpha_n}.$$

\noindent We also use the following notations:

 $$\E (U) \text { is the space of infinitely smooth complex functions
 on an  open set } U\subset  \R^n,$$

 $$\E_0 \text { is the ring  of germs at  the origin of such  functions},$$

 $$ \F_n \, \, (\text { or simply } \F ) \text { is the
ring of formal power series in } n \text { variables},$$

 $$\oka_n \, \,  (\text { or simply } \oka ) \text { is the   ring of
convergent power series }.$$
\medskip

Let $M=\{m_0, m_1,
m_2, \dots \}$ be an increasing sequence of positive real numbers with $m_0=1.$

\begin{Definition}
We say that a function $f\in \E(U)$ belongs to the Denjoy-Carleman class $C^{M}_n(U)$ if there are  constants $A,B$ depending on $f,$  such that
$$\forall\, x \in U,\, \forall\, \alpha \in \N^n \quad |D^\alpha f(x)|\leq |\alpha|!\, A B^{|\alpha|}\, m_{|\alpha|}
.$$
\end{Definition}

\begin{Definition}
The ring of germs $ C^{M}_{n,0} $ is the inductive limit of $C^{M}_n(U)$, for $U$ in the family of neighborhoods of the origin.

\end{Definition}

\begin{Definition}
We say that a formal power series $$f = \sum_{\alpha \in \N^n} \frac{a_\alpha}{\alpha!} z^\alpha\in \F_n$$ belongs to the Denjoy-Carleman class $\F^{M}_n$ if there are
 constants $A,B$   such that
$$\forall\, \alpha \in \N^n\quad  |a_\alpha|\leq |\alpha|!\, A B^{|\alpha|}\, m_{|\alpha|}
.$$
\end{Definition}

In order to get classes of functions with some good structural properties,
we must impose some conditions on the sequence $M=\{ m_n\}$.

\smallskip

\begin{equation}\label{logconv}
\displaystyle \frac{m_{j+1}}{m_j}\leq\frac{m_{j+2}}{m_{j+1}} \hbox{\rm \,  for all }j \geq 0 \, \, \, \hbox { ({\em logarithmic convexity}) }.
\end{equation}

and

\begin{equation}\label{roots}
 \displaystyle \sup_j \sqrt[j]{\frac{m_{j+1}}{m_j}}< \infty
\end{equation}

We summarize the properties of $C^{M}_{n,0} $ when the sequence $M$ verifies
conditions (\ref{logconv}, \ref{roots}). We omit the $n$ and write simply $C^{M}_0 .$
For the proof and further references, one can consult \cite{thilliez}.

\begin{Proposition}
\label{DC}

The following properties hold for $C^M_0:$
\begin{itemize}
\item $C^{M}_0$ is a local ring containing $\oka$ with maximal ideal $$\m = \{f \in C^M_0 \, : \, f(0)=0\} = (x_1,\ldots x_n)C^{M}_0 .$$
\item $C^M_0$ is closed under division by coordinates functions $x_1, \dots, x_n.$
\item $C^M_0$ is closed under composition.
\item $C^M_0$ is closed under differentiation.
\item The Implicit Function Theorem holds for $C^M_0.$
\item The Inverse Function Theorem holds for $C^M_{1,0}.$
\end{itemize}
\end{Proposition}

\begin{Definition}
A Denjoy-Carleman class $C^{M}_0,$ properly containing the ring $\oka$, is called {\em quasi-analytic} if the Taylor map $T_0: C^{M}_0  \to \F^M$
 $$T_0 f = \sum_{\alpha \in \N^n} \frac{f^{(\alpha)}(0)}{\alpha !} z^\alpha$$
 sending each germ to its Taylor series at $0$ is injective. 

\end{Definition}

By the Denjoy-Carleman Theorem, the Taylor map is injective if and only if

\begin{equation}
\label{quasi}
\:\sum_{k=0}^\infty \: \frac{m_k}{(k+1) \, m_{k+1}} = \infty.
\end{equation}

Also $C^{M}_0 \supsetneq \oka$ if and only if

\begin{equation}
\label{biggerthanra}
\lim_{j \rightarrow \infty} \sqrt[j]{m_j} = \infty.
\end{equation}

This last condition combined with logarithmic convexity is easily proved to be equivalent to the following one:

\begin{equation}
\label{biggerthanra2}
m_{j+1} = \alpha_j m_j \,\,\hbox{\rm with}\, \lim_{j \rightarrow \infty} \alpha_j = \infty.
\end{equation}

We will use both in what follows.

Finally, we recall that the inclusion $C^M_0 \subset C^{\widetilde M}_0$ is equivalent to the condition

$$\displaystyle \sup_{n\geq 0}  \sqrt[n]{\frac {m_n}{\widetilde m_n}}<\infty.$$
\medskip

{\it Basics of Weierstrass Division and  Preparation}.
Assume $n \geq 2.$
\smallskip

We say that $\varphi\in C^M_{n-1,0}[x_n],$ a monic polynomial in $x_n$ of degree $d$,
$$ \varphi(x)={x_n}^d+a_1(x'){x_n}^{d-1}+\ldots+a_d(x'),$$
where $x=(x',x_n)$ and  $a_j \in C^M_{n-1,0},$
is a \emph{distinguished Weierstrass polynomial in $x_n$} if
$a_j(0)=0$, for all $1 \leq j \leq d.$

\smallskip
Such  a  polynomial $\varphi$  is  called  \emph{hyperbolic} in $x_n$ if  there
exists a neighborhood  $U$ of $0$ in ${\R}^{n-1}$  such that $\forall
x' \in U$,  all the roots of $\varphi(x',\cdot)$  are real; otherwise,
$\varphi$ is called \emph{non-hyperbolic} in $x_n.$

\smallskip
A  germ $f  \in C^{M}_{n,0}$  is \emph{regular  in  $x_n$ of
  order  d}  if  there  exists   a  unit  $u$  in  the  ring
$C^{M}_{n,0}$ such that $f(0,  x_n)=u(0,x_n) \, x^d_n.$

Note this is
equivalent  to  saying  that  $$f(0,0)=\frac{\partial}{\partial  x_{n}}
f(0,0)=\dots  =\frac{\partial}{\partial x^{d-1}_{n}}  f(0,0)=0,$$ while
$$\frac{\partial}{\partial x^{d}_{n}} f(0,0)\neq 0.$$

\smallskip
 A germ $f \in
C^{M}_{n,0}$  is \emph{strictly  regular  in $x_n$  of  order d}  if
$$\frac{\partial}{\partial  x^{d}_{n}}  f(0,0)\neq   0,  \text { but  for  any }
\alpha   \in  \N^n,\,  |\alpha|  \leq  d-1, \quad D^{\alpha} f(0,0)=0.$$

\begin{Remark}
  Note that any germ of a  smooth function and any formal power series
  can be made strictly regular with respect to a chosen variable via
  a linear change of variables.
\end{Remark}

\begin{Definition}
  A ring $R_n$ of germs  of infinitely differentiable functions in $n$
  variables  at $0$  is said  to satisfy  the  \emph{Weierstrass Preparation
  Theorem} if for any $d \geq 1$ and any $f \in R_n$ that is regular of
  order $d$ in  $x_n$  \emph{can be prepared} in $R_n$, i.e., there exist a unit $u \in \R_n$ and
  a   distinguished   Weierstrass   polynomial  $\varphi   \in
  R_{n-1}[x_n]$ of order $d$, such
  that $$f = u \varphi.$$
\end{Definition}

\begin{Definition}
  We  say  that  the  \emph{Weierstrass Division  Property}  holds  in
  $C^{M}_{n,0}$ for a function $f$ that  is regular in $x_n$ of order $d$
  if  for every $g  \in C^{M}_{n,0},$  there exist  $q \in  C^{M}_{n,0}$ and  $h \in C^{M}_{n-1,0}[x_n] $, a polynomial of degree $\le (d-1)$ in $x_n,$
   such that
$$g=f q + h.$$ 
\end{Definition}

\begin{Remark}
  It is  classically known  that in the  ring of  holomorphic function
  germs, the ring of real analytic function germs,  and the
  ring  of  smooth  germs, the Weierstrass  Preparation Theorem holds true and they also have  the  Weierstrass Division Property.
  \end{Remark}

  \begin{Remark} What are the relations between the two Weierstrass properties?

   The Weierstrass  Division  Property for  an element  $f$
  regular in $x_n$  of order $d$ in a ring  of germs $R_0$ implies
  that the  same $f$ can  be prepared. Indeed,  it is enough  to divide
  $x_n^d$ by  $f,$ and we obtain $x_n^d=  q f +r$ with $q,r \in R_0.$ Comparing  orders we see
  that $q$  must be a  unit and the  remainder $r$ is a  polynomial in
  $x_n$ of degree  $d-1$ whose coefficients must  vanish at $0$. Hence
  $f=q^{-1}(x_n^d-r)$ is prepared.

  \medskip
    We do not know if the Weierstrass Preparation Theorem for a ring $R_0$ implies the Weierstrass Division Property for this ring.
A natural way to establish this hits an obstacle: assume the  Weierstrass
  Preparation  Theorem  holds  true  in  $R_0$  and  take  $f,g\in
  R_0.$ We  can assume both are  regular in the  same variable, so
  that there exist units $u,v$  and Weierstrass polynomials  $p,q$ such
  that $f=up,$  $g =vq.$ Perform the Euclidean division between the two Weierstrass polynomials $p$ and $q,$ we get $p= aq +r$ and thus $f = up = v^{-1} (ua)vq + ur = v^{-1}ua g + ur.$ There is no guarantee, however, that $ur$ can be reduced to a polynomial.
\end{Remark}

From    Childress   \cite{childress}    and    Chaumat-Chollet   \cite
{chaumatchollet}, we  know that for $n\geq 2$  a Weierstrass polynomial
$g\in  C^{M}_{n,0}$  has the Weierstrass  Division Property if  and only if it  is a
hyperbolic  polynomial. In  particular, the  Weierstrass Division  Property
does not  hold in general in $ C^{M}_{n,0}$ (the reader can find specific examples of functions for which it fails in \cite{childress} and \cite{thilliez}), but this fact does not
imply the failure of the Weierstrass Preparation Theorem.

\medskip
Note that
the Weierstrass  Division  Property, and,  hence, the Weierstrass Preparation, both hold in the ring $\F^M$, and, moreover, this ring  is  Noetherian. This was
proved by Chaumat et Chollet in \cite {formalnoetherian}. Yet, the Taylor morphism $C^{M}_0 \to \F^M$ is not surjective as proven long ago  by Carleman in \cite{carlemanb}, so the better algebraic properties of $\F^M$ do not automatically transfer to $C^{M}_0.$

\begin{Remark} What about the Weierstrass Division and Preparation in \emph{non-quasi-analytic} Denjoy-Carleman rings $C^{\{m_k\}}_{n,0}$?
 M. Bronshtein \cite{bronshteinnext} 
 has shown that the Weierstrass Division Property holds  \emph{weakly}: for any $d\in \N$, for every $f\in C^{\{m_k\}}_{n,0},$ regular of order $d$ in $x_n,$ and for every function $g  \in C^{\{m_k\}}_{n,0},$
  there exist a function   $q$ from the \emph{wider} Denjoy-Carleman ring $C^{\{m_{dk}\}}_{n,0}$ and
  a polynomial $h \in C^{\{m_{dk}\}}_{n-1,0}[x_n]$ over this wider ring, such that  $$g=f q + h.$$

From this result, he deduced that, given $d\in \N$, every $f\in C^{\{m_k\}}_{n,0}$ regular of order $d$ in $x_n$ can be  prepared in this wider ring $C^{\{m_{dk}\}}_{n,0}$, meaning that there exist  a unit $u\in C^{\{m_{dk}\}}_{n,0}$ and a distinguished  Weierstrass   polynomial    $\varphi   \in C^{\{m_{dk}\}}_{n-1,0}[x_n]$ of degree $d$, such
  that $$f = u \varphi.$$
  %\end{Remark}

  So  a \emph{controlled} widening of the ring (depending upon the regularity of $f$) saves the day for both Weierstrass theorems. This widening is optimal.

 This result suggests that in the non-quasi-analytic case one should consider a wider ring,
  $$\widetilde{C}^M(U) = \bigcup_{d\in \N} C^{\{m_{kd}\}}(U).$$
   In this new class,  the Weierstrass Division Theorem holds and therefore every function can be prepared.
 %\end{Remark}

\medskip

When we started working on the quasi-analytic case, we expected to find a wider quasi-analytic ring containing all the objects we were looking for. Instead, it turned out that the situation is drastically different in the quasi-analytic case.

%\end{Remark}

  \medskip
  Our main result asserts:
  \medskip

   {\it Let $C^M_{n,0},\, \, n\ge 2$ be a quasi-analytic Denjoy-Carleman ring. Choose ANY  wider quasi-analytic Denjoy-Carleman ring $C^{\widetilde{M}}_{n,0}$. There exists $f\in C^M_{n,0},$ regular of order $2$ in $x_n,$ which cannot be   prepared in $C^{\widetilde{M}}_{n,0}.$}
  \medskip

  So we say that the Weierstrass Preparation Theorem \emph{strongly fails} in quasi-analytic Denjoy-Carleman rings in dimensions $\ge 2.$
  \end{Remark}

\section{Extension of Functions from Denjoy-Carleman Classes}

Given a quasi-analytic local ring  $C^M_0 \supsetneq \oka,$ we will need to find a wider ring $C^{\widetilde M}_0$ verifying two conditions pulling in opposite directions.
On one hand, we require $C^{\widetilde M}_0$ to be quasianalytic, that is to
verify \ref {quasi}.  On the other hand, we ask that $\displaystyle \lim_ {n
  \to \infty} \sqrt[n]{\frac{\widetilde m _n}{m_n}}= \infty$, which in
particular implies $C^M_0 \subsetneq C^{\widetilde M}_0 $. This is done in the following lemma:

\begin{Lemma}\label{WEP}
For any quasi-analytic local ring $C_0^M$  there exists a quasi-analytic local
ring $C_0^{\widetilde M}$ with $C_0^M \subset C_0^{\widetilde M}$ and  $\displaystyle \lim_ {j \to \infty} \sqrt[j]{\frac{\widetilde m _j}{m_j}}= \infty$.
\end{Lemma}
\begin{proof}
Since $C_0^M$ is quasi-analytic, we have $\displaystyle \sum_{j\geq 1}\frac{m_{j-1}}{j\cdot m_j}=\infty$. Put $\displaystyle \alpha_j =\frac{m_j}{m_{j-1}}$  and  $\displaystyle \mu_j =\sum_{k=1}^j\frac{1}{k\,\alpha_k}$. So $\displaystyle \lim_ {j \to \infty} \mu_j = \infty$.
Define $\beta_j= \alpha_j \,\sqrt{\mu_j}$, so that $\displaystyle \lim_ {j \to \infty} \frac {\beta_j}{\alpha_j}
=\infty$.

Then $\displaystyle \sum\frac{1}{j\,\beta_j}$ diverges. Indeed, its partial sums verify
$$\sum  _{k= 1}^j\frac{1}{k\,\beta_k}=\sum _{k= 1}^j\frac{1}{k\,\alpha_k\,\sqrt{\mu_k}}>\frac{1}{\sqrt{\mu_j}}\sum _{k= 1}^j\frac{1}{k\,\alpha_k}=\sqrt{\mu_j}.$$

Now let us define $ \widetilde m_j = \beta_j \widetilde m_{j-1}$, $\widetilde m_0=1$. We obtain that
\begin{itemize}
\item $C_0^{\widetilde M}$ is quasi-analytic by construction.
\item $\displaystyle \frac{\widetilde m_j}{m_j} =\frac{\beta_j}{\alpha_j}\frac{\widetilde m_{j-1}}{m_{j-1}}$, so $\displaystyle \lim_ {j \to \infty} \frac {\beta_j}{\alpha_j}
=\infty$ gives $\displaystyle \lim_ {j \to \infty} \sqrt[j]{\frac{\widetilde m_j}{m_j}}= \infty.$
\item The previous limit implies $\displaystyle \sup_{j\geq 0} \left (\frac {m_j}{\widetilde m_j}\right )^{\frac{1}{j}}<\infty$ as needed for $C_0^{ M}\subset C_0^{\widetilde M} .$
\end{itemize}

\end{proof}

\begin{Definition} For $U = [0,\varepsilon) \subset \R_+$ let $$C_{+}^{M}(U) = \{ f\in \E(U): \, \exists A,B > 0
\quad \forall\, k\in \N, \, \,  \forall\, x \in U\quad |D^k f(x)| \le k! AB^k m_k\}.   $$

The space $C_{+,0}^M$ is the space of germs at $0$ of such functions.
\end{Definition}

It is very natural to ask whether every germ  in $C^M_{+,0}$ is the restriction of some germ  in
$C^{M}_{1,0}$, i.e., if $$C^M_{+,0} \subset C^{M}_{1,0}|_{\R_+},$$ or, more generally,  if
\begin{equation}\label{funcext}C^M_{+,0} \subset C^{\widetilde{M}}_{1,0}|_{\R_+}\end{equation} for a wider class $C^{\widetilde M}_{1,0}$, with $\widetilde{M}$ depending only upon $M$.

If this is true, we say, respectively, that $ C_{+,0}^M $  has the {\sl  strong extension property} or the {\sl weak extension property}.

\begin{Remark} For a non-quasi-analytic class $C^M_{+,0}$ there are various descriptions of (generally, wider) classes $C^{\widetilde{M}}_{1,0}$ satisfying (\ref{funcext}); see for example the papers mentioned in the introduction. So the weak (but usually not the strong) extension property  holds for non-quasi-analytic Denjoy-Carleman classes.
\end{Remark}

\medskip
The situation of quasi-analytic Denjoy-Carleman classes is quite different. The crucial difference from  the previous case is  that if a germ from a quasi-analytic class $C^M_{+,0}$ extends to a germ from another quasi-analytic class $C^{\widetilde{M}}_{1,0}$, then this extension  is unique due to the quasi-analyticity. In fact, this uniqueness of extension often prevents the existence of an extension.

\medskip
In this framework, there are interesting  results
by Langenbruch \cite{langenbruch} and Thilliez \cite{thillieznonext}; see also Nowak \cite{nowakcounter}. From there we extract the following statement:

\begin{Theorem}\label{nowak}  Let $C^M_{1,0}$ and $C^{\widetilde M}_{1,0}$  be a  quasi-analytic local
  rings. If $C^M_{1,0}$ properly contains the ring  of analytic germs $\oka_1$
  and $C^M_{1,0} \subset C^{\widetilde M}_{1,0},$ then $C^{M}_{+,0}\setminus \left(C^{\widetilde M}_{1,0}|_{\R_+}\right) \neq
  \emptyset.$
\end{Theorem}
\smallskip

So Theorem \ref{nowak} asserts that \emph{any} quasi-analytic local ring $C^M_{1,0}$ containing $\oka$ fails the weak extension property, i.e.,
no matter how wide is the quasi-analytic class $C^{\widetilde M}_{1,0}$ we choose, there still exist germs in $C^{M}_{+,0},$ which do  not extend to germs from $C^{\widetilde{M}}_{1,0}.$

 In the last section, we reprove this result by a very explicit construction of non-extendible germs, which have useful additional properties that we plan to use in our further research. Our construction yields another result not implied in an evident way by Theorem \ref{nowak} that we need in the next section.
 Namely, the following is an easy consequence of Theorem \ref{misha}:

\begin{Theorem}\label{misha2}
Let $C^M_{1,0}, C^N_{1,0}$ be quasi-analytic Denjoy-Carleman   rings. Assume that $C^M_{1,0}$  properly contains $\oka_1.$
 Then there exists a quasi-analytic ring $C^K_{1,0}$ such that $C^N_{1,0} \subset C^{K}_{1,0}$, and
 $$\left(C^{M}_{+,0}\cap C^{K}_{1,0}|_{\R_+}
\right) \setminus C^N_{1,0}|_{\R_+} \neq \emptyset.$$
\end{Theorem}

\begin{proof}
First, let us note that, due to Lemma \ref{WEP}, there exists  a sequence $K =\{k_p\},$ such that $C^K_{1,0}$ is a quasi-analytic ring, $C^N_{1,0} \subset C^{K}_{1,0}$ and
\begin{equation}\label{infty} \lim_ {j \to \infty} \left ( \frac{k_j}{n_j}\right )^{\frac{1}{j}}= \infty.\end{equation}

 Next, let $f$ be a function as in Theorem \ref{misha} for $\widetilde{M} = K$. Let us verify that $$f|_{\R_+} \in \left(C^{M}_{+,0}\cap C^{K}_{1,0}|_{\R_+}
\right) \setminus C^N_{1,0}|_{\R_+}.$$

According to Theorem \ref{misha}, $f\in C^{K}_{1,0},$ and $f|_{\R_+} \in C^{M}_{+,0}.$
Let us show that $f\notin  C_{1,0}^N,$ which will prove the Theorem, since the restriction map $C_{1,0}^K \rightarrow C_{1,0}^K|_{\R_+}$ is injective.
\medskip

Assume the opposite, $f\in C_{1,0}^N.$ Then there exist $A, B > 0$  such that $ |f^{(p)}(x)|\leq p! A B^p n_p $ for all $p\in \N$ and for all points $x$ in a neighborhood of $0.$

By Theorem \ref{misha} we   have
$$\left| f^{(p_j)}(x_{j})\right|   \geq  p_j!\,k_{p_j}.$$

Therefore we get
$$\forall\, j\quad  p_j!\,k_{p_j} \leq p_j!\,A B^{p_j}  n_{p_j}.$$
 So for any $j$
$$ \sqrt[p_j]{\frac{k_{p_j}}{n_{p_j}}}\leq A^{\frac{1}{p_j}} B,$$
which contradicts (\ref{infty}).
\end{proof}

\smallskip

\section{A Strong Failure of the Weierstrass Preparation Theorem}

\begin{Theorem}
\label{WPT} For any quasi-analytic local ring $C^M_{n,0},\, \, n\ge 2,$ and
 for any wider local quasi-analytic ring $C^{\widetilde{M}}_{n,0}$ there exists $g\in C^M_{n,0}$, regular of order 2, which cannot be prepared in $C^{\widetilde{M}}_{n,0}$.
\end{Theorem}

\begin{proof} It suffices to consider the case of dimension 2.
\smallskip

 By  Theorem \ref{misha2}, there exist a quasi-analytic ring $C^K_{1,0}$ and a function $f$ on a neighborhood $U\subset \R$ of $0$ such that
   $$f|_{U\cap \R_+} \in C^M(U\cap \R_+),\, \, f \in C^K(U), \text { but } f\notin C^{\widetilde{M}}(U).$$
\smallskip

We can assume  that  $f(0) = 0$ and $f'(0) > 0.$ If this
  is not the case,  it suffices to take $f(x) + kx  + b,$ which is in the
  same  class  as $f$  and  has  the  desired property  for  suitable
  constants  $k$ and $b.$
  
   Define $g(t,x)  =  f(t^2)-x.$ Since  $t^2 \geq  0$, by Proposition \ref{DC},
  $g(t,x)$ is a $C^M_{2,0}$ germ. Since  $f'(0)\ne 0,$  $g(t,x)$ is regular of order
  2 with respect to the variable $t.$ Let us show that $g$ \emph{cannot be prepared} in $C^{\widetilde{M}}_{2,0}.$ In the ring of formal power series $\F_2,$ there exist a second degree Weierstrass polynomial $P(t, x)$ and a unit $Q(t,x)$ such that the Taylor expansion $T_0 \, g$ of g at $0$ can be represented as $T_0 \, g(t,x)=P(t,x)Q(t,x),$ and this representation is unique. Since $g(t,x)=g(-t,x)$ and the Taylor morphism taking a germ to its Taylor expansion at $0$ is injective by quasi-analyticity, $P$ does not have degree $1$ terms in $t,$ so $P(t,x)=t^2-a(x).$ Setting $T_0 \, g(t,x)=T_0 \, f(t^2) -x =0$  and plugging in $x=T_0 \, f(t^2),$ one obtains $t^2=a(x),$ which implies $x= T_0 \, f(a(x)).$ Hence, as a formal series, $a=(T_0 \, f)^{-1}.$ The inverse series of the Taylor series of a function in some quasi-analytic class is the Taylor series of a function in the same class. Therefore, if $g$ could be prepared in $C^{\widetilde{M}}_{2,0},$ then $P(t,x)$ would have to be the formal power series of an element in $C^{\widetilde{M}}_{2,0},$ i.e. $a(x)$ would have to correspond to an element of $C^{\widetilde{M}}_{1,0},$ which is impossible because $ f\notin C^{\widetilde{M}}(U).$ 
\end{proof}

\section{An Explicit  Non-Extendable Function}\label{nonext}

\parindent = 0pt
In this section, we prove Theorem  \ref{misha} by constructing an explicit example of a function in question.

\begin{proof} Our bricks for constructing the function $f$ will be functions  $$\displaystyle g_n= \frac{A_n}{2^n(z_n-x)},$$ where  $z_n=x_n+iy_n$ with $\displaystyle x_n<0,$ $\displaystyle y_n>0,$ $\displaystyle \lim_{n\to \infty} x_n=\lim_{n\to \infty} y_n=0 ,$  and  $ {A_n}$ suitably chosen.
\medskip

To select appropriate $A_n$'s and $x_n$'s, consider functions $\varphi (\xi)$ and $\widetilde\varphi (\xi)$ related to one of the classic ways of gauging the growth of sequences $\{ m_t\}$ and  $\{ \widetilde{m}_t\}:$

$$\varphi(\xi) =\sup_{t>0} \frac{\xi^{t+1}}{{m}_t} $$

and

$$\widetilde\varphi(\xi) =\sup_{t>0} \frac{\xi^{t+1}}{\widetilde{m}_t}.$$

$\varphi (\xi)$ and $\widetilde\varphi (\xi)$ satisfy the following:

\begin{itemize}

\item They are increasing for $\xi >0$ since each is the supremum of monotonically increasing functions.
\item For every $\xi >0,$ we get $ \varphi(\xi)< \infty$, $\displaystyle \lim_{\xi\to \infty}\varphi(\xi)=\infty,$   $\xi^{t+1}\leq \varphi(\xi)\cdot  m_t$ and similarly $ \widetilde\varphi(\xi)< \infty$, $\displaystyle \lim_{\xi\to \infty}\widetilde\varphi(\xi)=\infty,$  $\xi^{t+1}\leq \widetilde \varphi(\xi)\cdot \widetilde m_t.$
\item Both are continuous, which combined with monotonicity implies they are increasing bijections of $(0, +\infty).$
\end{itemize}

All properties but continuity are easy to see. We establish the latter in Step 1 for $\widetilde\varphi (\xi),$ and obviously, the same argument applies for $\varphi(\xi).$

{\bf Step 1.} {\sl Call  $\displaystyle b_n=\frac {\widetilde m_n}{\widetilde m_{n-1}}$. Then for all $z\in [b_n,b_{n+1}],$ one has
$$\displaystyle \widetilde {m}_n = \sup_{\xi>0} \frac {\xi^{n+1}}{\widetilde\varphi(\xi)} = \frac{z^{n+1}}{\widetilde\varphi(z)}.$$ Furthermore, $\widetilde\varphi (\xi)$ is continuous.}

\begin{proof}

\begin{itemize}
\item [a)]$\displaystyle \forall\, j\leq n\quad \frac{z^{j-1}}{\widetilde m_{j-1}}\leq \frac{z^{j}}{\widetilde m_{j}}.$

Indeed, $\displaystyle \frac{z^{j}}{\widetilde m_{j}}= \frac{z^{j-1}}{\widetilde m_{j-1}}\cdot \frac{z}{b_{j}}\geq \frac{z^{j-1}}{\widetilde m_{j-1}}$
since $z\geq b_n\geq b_j$.

\item [b)] $\displaystyle \forall\,  j \geq n \quad \frac{z^{j+1}}{\widetilde m_{j+1}}\leq \frac{z^{j}}{\widetilde m_{j}}$.\

Indeed,
$ \displaystyle\frac{z^{j+1}}{\widetilde m_{j+1}}= \frac{z^{j}}{\widetilde m_{j}}\cdot \frac{z}{b_{j+1}}\leq \frac{z^{j}}{\widetilde m_{j}}$
since $ b_{j+1}\geq b_{n+1}\geq z.$
\end{itemize}

Thus
$$\displaystyle \forall\, z\in [b_n,b_{n+1}]\quad  \widetilde\varphi(z)= \sup_{k>0}\frac{z^{k+1}}{\widetilde m_{k}}=\frac{z^{n+1}}{\widetilde m_{n}}.$$ Therefore, $\widetilde\varphi$ is continuous.

From the above, it follows that
$$ \frac {z^{n+1}}{\widetilde\varphi(z)}= \widetilde m_n \geq \sup_{\xi>0} \frac{\xi^{n+1}}{\widetilde\varphi(\xi)}\geq \frac {z^{n+1}}{\widetilde\varphi(z)}$$
because $\xi^{t+1}\leq\widetilde\varphi(\xi)\widetilde m_t$ for all $\xi>0, \, t>0.$

Consequently, these inequalities are equalities.
\end{proof}

{\bf Step 2.} {\sl We can choose $A_n$ in such a way that $\displaystyle \left| g_n^{(p)}(x)\right|\leq p!\,\frac{1}{2^n}\,\widetilde m_p$ for every $x \in \R.$ 
}
\begin{proof}

For any $x\in \R,$ we have
$$\left|\frac{A_n}{(z_n-x)^{p+1}}\right| \leq \frac{|A_n|} {|y_n  |^{p+1}}$$
Taking $\displaystyle A_n= \frac{1}{\widetilde\varphi(y_n^{-1})}$ yields

\begin{equation}
\label{MtildeExpr} \left| g_n^{(p)}(x)\right|= p!\,\frac{A_n}{2^n|z_n-x|^{p+1}}\leq  p! \frac{1}{2^n} \frac {(y_n^{-1})^{p+1}}{\widetilde \varphi(y_n^{-1})}\leq  p!\,\frac{1}{2^n}\,\widetilde m_p
\end{equation}
since $\xi^{p+1}\leq \widetilde \varphi(\xi)\cdot \widetilde m_p.$

\end{proof}

{\bf Step 3.} {\sl A suitable choice of  $\{x_n\}$ implies  $\displaystyle \left| g_n^{(p)}(x)\right|\leq p! \frac{1}{2^n}\, m_p$ for every $x \in [0, \infty).$}
\begin{proof}
Observe that for $x\geq0,$ we have
$$\left|\frac{A_n}{(z_n-x)^{p+1}}\right| \leq \frac{|A_n|} {|x_n  |^{p+1}}=\frac {|x_n ^{-1} |^{p+1}}{\widetilde\varphi (y_n^{-1})}.$$

Since both $\varphi(\xi)$ and $\widetilde\varphi(\xi)$ are continuous bijections on $(0, +\infty),$ we can take $\overline x_n$ such that

$$\varphi(\overline {x}_n^{\,-1})= \widetilde \varphi(y_n^{-1}).$$
Letting $x_n = -\overline x_n,$ we obtain

\begin{multline}
$$\displaystyle \left| g_n^{(p)}(x)\right|\leq  p! \frac{1}{2^n}\cdot\frac{|A_n|} {|x_n  |^{p+1}}=p! \frac{1}{2^n}\cdot\frac {|x_n ^{-1} |^{p+1}}{\widetilde\varphi (y_n^{-1})}=p!\frac{1}{2^n}\frac {|x_n ^{-1} |^{p+1}}{\varphi (\left|x_n\right|^{-1})}\leq p! \frac{1}{2^n}\, m_p.$$ \label{MExpr}
\end{multline}
\end{proof}

In order to have a function  $f$ verifying property (3), we have to extract a suitable subsequence $\{g_{n_j}\}\subset \{g_{n}\}$ and we must be careful in choosing the sequence $\{y_n\}$.

Let  $\{n_j\}$ any subsequence of $\N.$ We shall estimate the  $n_j^{th}$ derivative of the function $\displaystyle S= \sum_t g_{n_t}$ at the point $x_ {n_j}$ as follows:
\begin{multline}
$$
\label{Sjexpr}
\left|S^{(n_j)}(x_{n_j})\right|=\left|g_{n_j}^{(n_j)}(x_{n_j}) +\sum_{t\neq j}g_{n_t}^{(n_j)}(x_{n_j})\right|\\\geq \left| \left|g_{n_j}^{(n_j)}(x_{n_j})\right|- \sum_{t\neq j}\left|g_{n_t}^{(n_j)}(x_{n_j})\right|\right|.$$
\end{multline}

On the right side  of the inequality above, there are three terms, namely
\begin{enumerate}
\item [a)] $\displaystyle \left|g_{n_j}^{(n_j)}(x_{n_j})\right|= n_j! \frac {(y_{n_j}^{-1})^{n_j+1}}{2^{n_j}\widetilde\varphi(y_{n_j}^{-1})}$
\item [b)] $\displaystyle \sum_{t< j}\left|g_{n_t}^{(n_j)}(x_{n_j})\right|$
\item [c)] $\displaystyle \sum_{t> j}\left|g_{n_t}^{(n_j)}(x_{n_j})\right|$
\end{enumerate}

We shall evaluate them in subsequent steps.

{\bf Step 4.} {\sl Choosing $y_n^{-1}\in [b_n, b_{n+1}],$ we have}
$$|g_n^{(n)}(x_n)| = 2^{-n} n!\,\widetilde m_n \quad \forall n\in \N$$

\begin{proof}
$\displaystyle |g_n^{(n)}(x_n)|= n!\,\frac {(y_{n}^{-1})^{n+1}}{2^{n}\widetilde\varphi(y_{n}^{-1})}=2^{-n} n!\,\widetilde m_n. $

\end{proof}

{\bf Step 5.} {\sl $\forall\, l\  \exists\, N>l \mbox{ such that }  \forall\, t\leq l\ \mbox {and}\  \forall\,n>N$\,
 the following inequality holds}
$$\frac{(y_t^{-1})^{n+1}}{\widetilde\varphi(y_t^{-1})} <2^{-(n+2)}\widetilde m_n$$

\begin{proof}
$$\frac{(y_t^{-1})^{n+1}}{\widetilde\varphi(y_t^{-1})}= \frac{(y_t^{-1})^{t+1}}{\widetilde\varphi(y_t^{-1})}\cdot (y_t^{-1})^{n-t}= \widetilde m_t\,(y_t^{-1})^{n-t}=\widetilde m_n \,\frac{y_t^{-1}}{b_{t+1}}\cdot\frac{y_t^{-1}}{b_{t+2}}\cdots\frac{y_t^{-1}}{b_n}\leq \widetilde m_n \,2^{-(n+2)}.$$
The second equality comes directly from  Step 1.

For the last inequality, take $N'$ such that $b_{N'}>4 b_{t+1}$. Note that with this assumption the factors up to index $N'-1$ are all bounded above by $1$, while the rest are bounded above by $1/4.$ Since there are  $n-N'+1$ of the latter factors, we obtain the upper bound $(1/4)^{n-N'+1}$.

Choosing $n>N=2N',$ we obtain that $N' < n/2,$ and

$$\displaystyle\frac{1}{4^{n-N'+1}}<\frac{1}{4^{(n/2) + 1}}= \frac {1}{2^{n+2}}.$$
\end{proof}

{\bf Step 6.} {\sl
  $$
|S^{(n_j)}(x_{n_j})|\geq n_j!\ \big [2^{-n_j}\widetilde m_{n_j}-\max_{t<j} \frac {(y_{n_t}^{-1})^{n_j+1}}{\widetilde\varphi(y_{n_t}^{-1})} - 2^{1-n_{j+1}} \widetilde m_{n_j}\big ]
$$}
\begin{proof}

From equation \eqref{Sjexpr}, we have

$$
\left|S^{(n_j)}(x_{n_j})\right|\geq  \left|g_{n_j}^{(n_j)}(x_{n_j})\right|- \sum_{t\neq j}\left|g_{n_t}^{(n_j)}(x_{n_j})\right| \geq $$

$$\geq  \left|g_{n_j}^{(n_j)}(x_{n_j})\right|- \sum_{t< j}\left|g_{n_t}^{(n_j)}(x_{n_j})\right|- \sum_{t> j}\left|g_{n_t}^{(n_j)}(x_{n_j})\right|\geq $$

$$\geq n_j!\,\left[\frac {(y_{n_j}^{-1})^{n_j+1}}{2^{n_j}\widetilde\varphi(y_{n_j}^{-1})} - \max_{t<j} \frac {(y_{n_t}^{-1})^{n_j+1}}{\widetilde\varphi(y_{n_t}^{-1})} - 2^{1-n_{j+1}}\max_{\xi>0}\frac{\xi^{n_j+1}}{\widetilde\varphi(\xi)}\right]= $$

$$=  n_j!\,\left[2^{-n_j}\widetilde m_{n_j} -\max_{t<j} \frac {(y_{n_t}^{-1})^{n_j+1}}{\widetilde\varphi(y_{n_t}^{-1})} - 2^{1-n_{j+1}}\widetilde m_{n_j}\right]$$
\end{proof}

{\bf Step 7}. Finally, we define the sequence $\{n_j\}$ recursively.
Suppose we have already defined $n_{j-1}.$ Define $n_{j}$ as $N$  in Step 5 with\,   $l= n_{j-1},$ \, also making sure that $$ n_{j}- n_{j-1}\geq 3.$$

{\bf Step 8}.
Now we have
$$\displaystyle
  \max_{t<j} \frac {(y_{n_t}^{-1})^{n_j+1}}{\widetilde\varphi(y_{n_t}^{-1})}\leq 2^{-(n_j+2)}\widetilde m_{n_j}$$ and $$\displaystyle 2^{1-n_{j+1}}\widetilde m_{n_j}\leq 2^{-2-n_j}\widetilde m_{n_j}.$$
Putting it all together, we obtain the inequality
$$\left| S^{(n_j)}(x_{n_j})\right| \geq n_j!2^{-n_j}\left( 1 -\frac{1}{4}-\frac{1}{4}\right)\widetilde m_{n_j}= \frac{1}{2}\, n_j!\,2^{-n_j}\widetilde m_{n_j}$$ for all $j.$

\medskip

Since our construction always gives  $$\lim_{n\rightarrow \infty} A_n = 0 \text{ and }
\lim_{n\rightarrow \infty} x_n = 0,$$ we have uniform convergence of the series $\sum g_n(x)$
 on  a complex neighborhood of any set $\R\setminus (-\varepsilon, \varepsilon),$ which guarantees
 the analyticity of $f (x) = 2S(2x)$ on any
such set. Note that Equation~\eqref{MtildeExpr} shows $f \in C^{\widetilde M}(\R),$ Equation~\eqref{MExpr} gives $f \in C^{M}[0,\infty),$ and the conclusion of Step 8 establishes property (3). Therefore, $f$ satisfies all properties in Theorem \ref{misha}. 
\end{proof}

\bibliographystyle{plain}
\bibliography{NonWPT}

\begin{thebibliography}{10}

\bibitem{acampo}
Norbert A'Campo.
\newblock Divisions dans des sous-anneaux de {R{$[[X_{1},\cdots, X_{n}]]$}}.
\newblock {\em J. Math. Pures Appl. (9)}, 46:279--298, 1967.

\bibitem{borel1}
\'{E}mile Borel.
\newblock Sur la g\'{e}n\'{e}ralisation du prolongement analytique.
\newblock {\em C. R. Acad. Sci. Paris}, 130:1115--1118, 1900.

\bibitem{borel2}
\'{E}mile Borel.
\newblock Sur les s\'{e}ries de polyn\^{o}mes et de fractions rationnelles.
\newblock {\em Acta Math.}, 24:309--387, 1901.

\bibitem{bronshteinnext}
M.~D. Bronshte{\u\i}n.
\newblock Division with a remainder in spaces of smooth functions.
\newblock {\em Trudy Moskov. Mat. Obshch.}, 52:110--137, 247, 1989.

\bibitem{carlemanb}
Torsten Carleman.
\newblock {\em Les fonctions quasi-analytiques}.
\newblock Gauthiers Villars, Paris, 1926.

\bibitem{carleson}
Lennart Carleson.
\newblock On universal moment problems.
\newblock {\em Math. Scand.}, 9:197--206, 1961.

\bibitem{formalnoetherian}
Jacques Chaumat and Anne-Marie Chollet.
\newblock Caract\'erisation des anneaux noeth\'eriens de s\'eries formelles \`a
  croissance control\'ee. {A}pplication \`a la synth\`ese spectrale.
\newblock {\em Publ. Mat.}, 41(2):545--561, 1997.

\bibitem{chaumatchollet}
Jacques Chaumat and Anne-Marie Chollet.
\newblock Division par un polyn\^ome hyperbolique.
\newblock {\em Canad. J. Math.}, 56(6):1121--1144, 2004.

\bibitem{childress}
C.~L. Childress.
\newblock Weierstrass division in quasianalytic local rings.
\newblock {\em Canad. J. Math.}, 28(5):938--953, 1976.

\bibitem{ehrenpreis}
Leon Ehrenpreis.
\newblock {\em Fourier analysis in several complex variables}.
\newblock Pure and Applied Mathematics, Vol. XVII. Wiley-Interscience
  Publishers A Division of John Wiley \& Sons, New York-London-Sydney, 1970.

\bibitem{langenbruch}
Michael Langenbruch.
\newblock Extension of ultradifferentiable functions.
\newblock {\em Manuscripta Math.}, 83(2):123--143, 1994.

\bibitem{lyubichtkachenko}
Ju.~I. Ljubi{\v{c}} and V.~A. Tka{\v{c}}enko.
\newblock The reconstruction of infinitely differentiable functions from the
  values of their derivatives at zero.
\newblock {\em Teor. Funkci\u\i\ Funkcional. Anal. i Prilo\v zen. Vyp.},
  9:134--141, 1969.

\bibitem{mityagin}
B.~S. Mitjagin.
\newblock An infinitely differentiable function with the values of its
  derivatives given at a point.
\newblock {\em Dokl. Akad. Nauk SSSR}, 138:289--292, 1961.

\bibitem{nowakcounter}
Krzysztof~Jan Nowak.
\newblock A counter-example concerning quantifier elimination in quasianalytic
  structures.
\newblock \textsl{{P}reprint.}, May 18 2012.

\bibitem{thillieznonext}
Vincent Thilliez.
\newblock On the {N}on-extendability of {Q}uasianalytic {G}erms.
\newblock \textsl{{P}reprint.} arXiv:1006.4171v1, [math.CA] 21 Jun 2010.

\bibitem{thilliez}
Vincent Thilliez.
\newblock On quasianalytic local rings.
\newblock {\em Expo. Math.}, 26(1):1--23, 2008.

\bibitem{wahde}
G{\"o}sta Wahde.
\newblock Interpolation in non-quasi-analytic classes of infinitely
  differentiable functions.
\newblock {\em Math. Scand.}, 20:19--31, 1967.

\bibitem{zobin}
N.~M. Zobin.
\newblock Extension and representation theorems for {G}evrey type spaces.
\newblock {\em Dokl. Akad. Nauk SSSR}, 212:1280--1283; letter to the editors,
  ibid. 216\ (1974), viii, 1973.

\bibitem{zobin2}
Nahum Zobin.
\newblock Szeg{\H o}-type extremal problems.
\newblock In {\em Voronezh {W}inter {M}athematical {S}chools}, volume 184 of
  {\em Amer. Math. Soc. Transl. Ser. 2}, pages 253--263. Amer. Math. Soc.,
  Providence, RI, 1998.

\end{thebibliography}

\end{document}